\documentclass[a4paper, twoside,12pt]{article}
\usepackage{fancyhdr}
\usepackage{fnpos}
 \usepackage[english]{babel}        
\usepackage[T1]{fontenc}          
\usepackage{graphicx}             
\usepackage{makeidx}
\usepackage{fancybox}
\usepackage{framed}
\usepackage{fancyhdr}
 \usepackage{pstricks,pst-plot,pstricks-add}
\usepackage[margin=1in]{geometry}
\usepackage{graphicx}
\usepackage{titlesec}
\usepackage{amsmath}
\usepackage{amsfonts}   
\usepackage{amssymb}    
\usepackage{amsthm}
\usepackage{dsfont}
\usepackage{mathtools}
\usepackage{verbatim}   
\usepackage{color}
\usepackage{listings}
\usepackage{graphicx}
\usepackage{amsmath}
\usepackage{amsmath,amssymb,color,inputenc,euscript,graphicx,psfrag}

\newtheorem{thm}{Theorem}[section]
\newtheorem{cor}[thm]{Corollary}
\newtheorem{lem}[thm]{Lemma}

\newtheorem{rem}[thm]{Remark}
\numberwithin{equation}{section}

\renewcommand{\thefootnote}

\newcommand\co{\operatorname{co}}

 {\markboth{\sl B.  Amri  and M. Gaidi }{\sl Oscillating Dunkl multiplier }

\author {B\'echir Amri and  Mohamed Gaidi}

\title{  $L^p$ estimates    for an  oscillating Dunkl multiplier  }

\date{}
 
\begin{document}
 \maketitle
\begin{center}
   Universit\'{e} Tunis El Manar, Facult\'{e} des sciences de Tunis,\\ Laboratoire d'Analyse Math\'{e}matique
       et Applications,\\ LR11ES11, 2092 El Manar I, Tunisie.\\
     \textbf{ e-mail:} bechir.amri@ipeit.rnu.tn,  gaidi.math@gmail.com
\end{center}
  \begin{abstract}
 In this paper, we study the $L^p$ boundedness of a class
   of oscillating multiplier operator for the Dunkl transform, $T_{m_\alpha}=\mathcal{F}_k^{-1}(m_{\alpha}\mathcal{F}_k(f))$ with
   $m(\xi)=|\xi|^{-\alpha}e^{\pm i|\xi|}\phi( \xi)$.  We obtain an $L^p$-bound result
    for  the corresponding maximal functions. As a specific applications, we give  an extension   of the $L^p$ estimate
  for the wave equation and  of   Stein's theorem for the analytic family of maximal spherical means \cite{Stein}.
 \end{abstract}
  \section{ Introduction and  Statement of Results}
  Let $\alpha > 0$ and $\beta> 0$, the
  oscillating  multiplier $T_{\alpha,\beta}$
is defined via Fourier transform by $\widehat{T_{\alpha,\beta}(f)}=m _{\alpha,\beta}\widehat{f}$, where $ m _{\alpha,\beta}=|\xi|^{-\alpha}e^{i|\xi|^{\beta}}\phi(\xi) $
 and $\phi$ is a $C^\infty$   function on $\mathbb{R}^n$ which vanishes near the origin
and is equal to  $1$ for all sufficiently large $\xi$.
 The study of their $L^p$ properties  going back to the works of  I. Hirschman \cite{H}  in the case $n=1$   and   S. Wainger \cite{Wa} in higher dimensions. Later on,  they have been extensively studied again by many authors   in several different contexts, see \cite{Feff, MIY, Per,Sj}.
  This paper is devoted to the study of $L^p$ boundedness of oscillating  multiplier in the context of Dunkl analysis.
 We will focus   on   the case  $\beta=1$,  because of the close connection to  the wave equation associated with the Dunkl Laplacian $\Delta_k$,
 and to the  spherical maximal function. The latter is already studied by  L. Deleaval \cite{DE}.
 In order to describe more precisely the results studied in this paper, we shall start by giving a brief summary of the Dunkl analysis.
\subsection{Background}
 Dunkl theory generalizes classical Fourier analysis on \,$\mathbb{R}^n$.
It started twenty years ago with Dunkl's seminal work \cite{D1}
and was further developed by several mathematicians.
We refer for more details to the articles \cite{D1, J1,R4} and the references cited therein.

\par Let $G\!\subset\!\text{O}(\mathbb{R}^n)$
be a finite reflection group associated to a reduced root system $R$
and $k:R\rightarrow[0,+\infty)$  a $G$--invariant function
(called multiplicity function).
Let $R^+$ be a positive root subsystem. The Dunkl operators  $D_\xi^k$ on $\mathbb{R}^n$ are
the following $k$--de\-for\-ma\-tions of directional derivatives $\partial_\xi$
by difference operators\,:
$$
D_\xi^k f(x)=\partial_\xi f(x)
+\sum_{\,\upsilon\in R^+}\!k(\upsilon)\,\langle\upsilon,\xi\rangle\,
\frac{f(x)-f(\sigma_\upsilon.\,x)}{\langle\upsilon,\,x\rangle}\,,
$$
where
\,$\sigma_\upsilon.\,x=
x-\frac{\langle\upsilon,\,x\rangle}{2\,|\alpha|^2}\,\upsilon$
\,denotes the reflection
with respect to the hyperplane orthogonal to $ \upsilon$. Here $\langle .,\, .\rangle$ is the usual Euclidean
inner product and $| \,.\,  |$ its induced norm.
The Dunkl operators are antisymmetric
with respect to the measure $w_k(x)\,dx$
with density
$$
w_k(x)=\,\prod_{\,\upsilon\in R^+}|\,\langle\upsilon,x\rangle\,|^{\,2\,k(\upsilon)}\,.
$$
The operators $\partial_\xi$ and $D_\xi^k$
are intertwined by a Laplace--type operator
\begin{eqnarray*}\label{vk}
V_k\hspace{-.25mm}f(x)\,
=\int_{\mathbb{R}^n}\hspace{-1mm}f(y)\,d\mu_x(y)
\end{eqnarray*}
associated to a family of compactly supported probability measures
\,$\{\,\mu_x\,|\,x\!\in\!\mathbb{R}^n\hspace{.25mm}\}$\,.
Specifically, \,$\mu_x$ is supported in the the convex hull $\co(G.x)\,.$
\par For every $y\!\in\!\mathbb{C}^n$\!,
the simultaneous eigenfunction problem
\begin{equation*}
D_\xi^k f=\langle y,\xi\rangle\,f
\qquad\forall\;\xi\!\in\!\mathbb{R}^n
\end{equation*}
has a unique solution $f(x)\!=\!E_k(x,y)$
such that $E_k(0,y)\!=\!1$, called the Dunkl kernel and is given by
\begin{equation*}\label{EV}
E_k(x,y)\,
=\,V_k(e^{\,\langle.,\,y\,\rangle})(x)\,
=\int_{\mathbb{R}^n}\hspace{-1mm}e^{\,\langle z,y\rangle}\,d\mu_x(z)
\qquad\forall\;x\!\in\!\mathbb{R}^n.
\end{equation*}
Furthermore this kernel has a holomorphic extension to $\mathbb{C}^n\times \mathbb{C}^n $
 and the following estimate hold\,: for  $ \;x, \;y\!\in\!\mathbb{C}^n,$
\begin{itemize}
\item[(ii)] $E_k(x,y)=E_k(y,x)$,
\item[(iii)] $E_k(\lambda x,y)=E_k(x,\lambda  y)$,\;  for $\lambda\in \mathbb{C}$,
\item[(iv)] $E_k(g. x,g.y)=E_k(x, y)$, for $g\in G$.
\end{itemize}
 In dimension $n\!=\!1$,
these functions can be expressed in terms of Bessel functions.
Specifically,
$$
E_k(x,y)= \mathcal{J}_{k-\frac12}(ixy)
+\frac{xy}{2\hspace{.25mm}k+1}\, \mathcal{J}_{k+\frac12}(ixy)\,,
$$
where
$$
 \mathcal{J}_\nu(z)\,=\;\Gamma(\nu\!+\!1)\;
 \sum_{\,n=0}^{+\infty}\;
\frac{(-1)^n}{n\hspace{.25mm}!\,\Gamma(\nu+n+1)}\;
 \left(\frac{z}{2} \right)^{2n}
$$
are normalized Bessel functions.
\par The Dunkl transform  is defined on $L^1(\mathbb{R}^n\!,w_k(x)dx)$ by
$$
\mathcal{F}_kf(\xi)={\textstyle c_k}
\int_{\mathbb{R}^n}\!f(x)\,E_k(x,-i\,\xi)\,w_k(x)\,dx\,,
$$
where
$$
c_k\,=\int_{\mathbb{R}^n}\!e^{-\frac{|x|^2}2}\,w(x)\,dx\,.
$$
We list some known properties of this transform\,:
\begin{itemize}
\item[(i)]
The Dunkl transform is a topological automorphism
of the Schwartz space $\mathcal{S}(\mathbb{R}^n)$.
\item[(ii)]
(\textit{Plancherel Theorem\/})
The Dunkl transform extends to
an isometric automorphism of $L^2(\mathbb{R}^n\!,w_k(x)dx)$.
\item[(iii)]
(\textit{Inversion formula\/})
For every $f\!\in\!\mathcal{S}(\mathbb{R}^n)$,
and more generally for every $f\!\in\!L^1(\mathbb{R}^n\!,w_k(x)dx)$
such that $\mathcal{F}_kf\!\in\!L^1(\mathbb{R}^n\!,w_k(\xi)d\xi)$,
we have
$$
f(x)=\mathcal{F}_k^2\!f(-x)\qquad\forall\;x\!\in\!\mathbb{R}^n.
$$
\item[(iv)] If $f$ is a radial function in $L^1(\mathbb{R}^n\!,w_k(\xi)d\xi)$ such that  $f(x)=\widetilde{f}(|x|)$, $x\in\mathbb{R}^n$. Then
$\mathcal{F}_k(f)$ is also radial and
\begin{equation}\label{rad}
    \mathcal{F}_k(f)(x)=d_k\int_0^\infty\widetilde{f}(s) \mathcal{J}_{\gamma_k+n/2-1}(s|x|)s^{2\gamma_k+n}ds; \quad x\in\mathbb{R}^n.
\end{equation}
where $d_k= 2^{-(\gamma_k+n/2-1)}/\Gamma(\gamma_k+n/2)$.
\end{itemize}
\par Let $x\in \mathbb{R}^n$, the
Dunkl translation operator $\tau_x$ is given for $f\in
L^2_k(\mathbb{R}^n,w_k(x)dx)$ by
\begin{eqnarray*}\label{dutr}
\mathcal{F}_k(\tau_x(f))(y)= \mathcal{F}_kf(y)\,E_k(x,iy), \quad
y\in\mathbb{R}^n.
\end{eqnarray*}
  In the case when $f(x)=\widetilde{f}(|x|)$ is a radial function in  $  \mathcal{S}(\mathbb{R}^n)$,  the Dunkl translation is represented by the following integral
\begin{eqnarray*}\label{trad}
\tau_x(f)(y)=
\int_{\mathbb{R}^{n}}\widetilde{f}\left( \sqrt{|y|^2+|x|^2+2\langle y,\eta\rangle}\;\right)d\mu_x(\eta),
 \end{eqnarray*}
This formula shows that the Dunkl translation operators can   be extended to all radial functions $f$ in $L^p (\mathbb{R}^n,w_k(x)dx)$, $1\leq p\leq \infty$  and the following holds
\begin{equation}\label{trp}
  \|\tau_x(f)\|_{p,k}\leq \|f\|_{p,k }\,.
  \end{equation}
where  $\|\,.\,\|_{p,k}$  denotes the  norm  of $L^p(\mathbb{R}^n,w_k(x)dx)$, $1\leq p<\infty$.
 \par We define  the Dunkl convolution product   for suitable functions $f$ and $g$ by
$$f*_kg(x)=\int_{\mathbb{R}^N} \tau_x(f)(-y)g(y)d\mu_k(y),\quad x\in\mathbb{R}^N.$$
We note that it is commutative and satisfies the following property:
\begin{eqnarray}\label{conv}
 \mathcal{F}_k(f*_kg)=\mathcal{F}_k(f)\mathcal{F}_k(g).
\end{eqnarray}
Moreover, the operator $ f \rightarrow f*_kg $ is bounded on $L^p (\mathbb{R}^n,w_k(x)dx)$
provide $g$ is a bounded radial function in $L^1(\mathbb{R}^n,w_k(x)dx)$. In particular we have the   the following Young's inequality:
\begin{equation}\label{Y}
 \|f*_kg\|_{p,k}\leq \|g\|_{1,k}\|f\|_{p,k}\;.
\end{equation}

\par  We  now come to the subject of this paper.
\subsection{Statement of Results}
Let  $\phi $ be  a $C^\infty$ radial function which vanishes for $|\xi|\leq 1/2$   and is equal to  $1$ for $|\xi| \geq 1$. Let $\alpha > 0$, we put
$$ m _{\alpha }(\xi)=|\xi|^{-\alpha}e^{\pm i|\xi| }\phi(  \xi  ).$$
The oscillating   Dunkl multiplier associated to $m_{\alpha }$
is defined  on $L^2(\mathbb{R}^n,w_k(x)dx)$ by
$$ T_{m_\alpha}(f)=\mathcal{F}_k^{-1}(m _{\alpha}\mathcal{F}_k(f)).$$
We will establish the following.
\begin{thm}\label{th1}
   $T _{m_\alpha}$  is a bounded  operator on $L^p(\mathbb{R}^n, w_k(x)dx)$   for all $1\leq p \leq\infty$ such that
  $$\alpha>(2\gamma_k+n-1)\left|\frac{1}{2}-\frac{1}{p}\right|.$$
\end{thm}
The proof  follows  closely the proof  given by S. Sjostrand in  \cite{Sj}.
This contains   $L^p$ estimates of   the solution operator for the Cauchy problem associated to Dunkl wave equation
$$\Delta_k u(x,t)=\partial_t^2 u(x,t),\quad u(x,y)=0,\quad\partial_t u(x,o)=f(x)$$
 where  $\Delta_k=\sum_{i=1}^n(T_{e_i}^k)^2$, which  is referred to as the    Dunkl-Laplace operator on $\mathbb{R}^n$.
Note that the solution to this problem has been already described in  \cite{said1} and  is  given by means of Dunkl's transform,
$$\mathcal{F}_k(u(.,t))(\xi)= \frac{\sin t|\xi|}{|\xi|}\mathcal{F}_k (f)(\xi).$$
As the   result, we have
  \begin{thm} \label{th2}
  For fixed $t$,  the linear operator $f\rightarrow u(x,t)$  is  bounded  operator on $L^p(\mathbb{R}^n, w_k(x)dx)$ to itself,  for $1\leq p \leq\infty$ such that
  $$\left|\frac{1}{2}-\frac{1}{p}\right|<\frac{1}{2\gamma_k+n-1}.$$
 \end{thm}
 Theorem \ref{th3} is therefore the  special case $\alpha=1$ of theorem \ref{th2}.  Indeed, we just write
 $$\frac{\sin(|\xi|)}{|\xi|} =\frac{e^{i|\xi|}}{2i|\xi|} \phi( \xi)-\frac{e^{-i|\xi|}}{2i|\xi|} \phi(\xi )+
 \frac{\sin(|\xi|)}{|\xi|}(1-\phi(\xi ))=a_1(\xi)+a_2(\xi)+a_3(\xi).$$
As $a_3$ is radial $C^\infty$-function with compact support, the Dunkl multiplier associated to $ a_3 $ is   the convolution operator with kernel $\mathcal{F}_k^{-1}(a_3)$ which is a  radial Schwartz function and then by Young's inequality it  is a bounded operator on $L^p(\mathbb{R}^n, w_k(x)dx)$ for all $1\leq p\leq\infty$ .
\begin{rem}
 In the statements  of theorems \ref{th1} and \ref{th2} nothing can be said about the endpoint of the range of $p$. Adaptation of  classical arguments as in \cite{Feff, MIY, Per} is not available,
 in the Dunkl  setting  the theory of Hardy space  $H^1$ and BMO space  are  not yet much elaborated. Also the $L^p$ theory of Dunkl multiplier is still ambiguous, since the latter is closely related to the generalized translation operators which  require  more information   than its known properties, in particular about their integral representations.
\end{rem}
The next main result is the boundedness of the  maximal  operator
  $$A_\alpha(f)(x)=\sup_{  t>0}|A_\alpha(f)(x,t)|, \quad $$
 where
 $$A_\alpha (f)(x,t)=\int_{ \mathbb{R}^n} \frac{e^{\pm it|\xi|}}{|t\xi|^k}\phi(t\xi)\mathcal{F}(f) (\xi)E_k(ix,\xi)w_k(\xi)d\xi.$$
\begin{thm} \label{th3}
 The maximal operator $A_\alpha$  is  bounded   on $L^p(\mathbb{R}^n, w_k(x)dx)$ , provide
 $$\alpha>( 2\gamma_k+n-1 )\left|\frac{1}{p}-\frac{1}{2}\right|+ \frac{1}{p}$$
 \end{thm}
Our arguments inspired by \cite{Sogge} for  the  study of the boundedness of certain Fourier integral operators.
   As an application,   we provide an
extension of the result in \cite{Stein}, obtaining boundedness for the  maximal function $\mathcal{M}_\alpha(f)(x) =\sup_{t>o} |\mathcal{M}_\alpha(f)(x,t)|$ where
 $$\mathcal{M}_\alpha(f)(x,t)= \int_{\mathbb{R}^n}\mathcal{J}_{\alpha+\gamma_k+n/2-1}(t|\xi|)\mathcal{F}_k(f)(\xi)w_k(\xi)d\xi$$
In particular, $\mathcal{M}_0$  reduces to the analogous of Stein's spherical maximal function,
   $$ \mathcal{M}_0(f)(x)=\sup_{t>0}\left|\int_{S^{n-1}} \tau_x(f)(ty)  d\sigma(y)\right|$$
 where   $S^{n-1}$ denotes the standard unit sphere in $\mathbb{R}^n$ and $d\sigma$  corresponds to the normalized surface measure. We obtain the following
   \begin{thm}\label{th4}
 The  maximal operator $\mathcal{M}_\alpha$ is  bounded   on $L^p(\mathbb{R}^n, w_k(x)dx)$ under the following conditions
 \begin{itemize}
   \item[(a)] $\alpha  > 1-2\gamma_k-n+(2\gamma_k+n)/p$, $\qquad $if  $ $\quad$  2\leq p <\infty,$
   \item[(b)] $\alpha> (2-2\gamma_k-n)/p$, $\qquad $ if $\quad$ $ 1< p \leq 2$.
 \end{itemize}
 \end{thm}
As a direct consequence of Theorem \ref{th4} we get the following result.
\begin{cor}Stein's spherical maximal function $ \mathcal{M}_0$ is  bounded   on $L^p(\mathbb{R}^n, w_k(x)dx)$, $n\geq 3$,  provide
  $$p> \frac{2\gamma_k+n}{2\gamma_k+n-1}\;.$$
 \end{cor}
\section{ Details of the Proofs}
We begin by choosing the Littlewood-Paley dyadic decomposition of unity   that we  shall use it all along this section. The existence of such a partition is standard, it is given by
  a   $C^\infty$-function  $\psi$   that is  supported in   $\{t\in \mathbb{R};\; 1/2\leq|t|\leq2\}$ and satisfies
  $$\sum_{-\infty}^\infty\psi(2^{-\nu} t)=1,\qquad t\neq 0.$$
We  can assume further that $0\leq \psi\leq 1$.
\par For convenience, we use the same notation $C$ to
denote the different constants in the different place.

  \subsection{ Proof of Theorem \ref{th1}}
  Without loss of generality  we will  prove the theorem \ref{th1} by taking  $\phi$    the function
   \begin{equation}\label{fi}
 \phi(\xi)=\widetilde{\phi}(|\xi|)=\sum_{\nu=0}^\infty\psi(2^{-\nu} |\xi|),
   \end{equation}
    which is clearly    a  $ C^\infty$-function on $\mathbb{R}^n$, with $\phi(\xi)=0$ for $|\xi|\leq 1/2$ and $\phi(\xi)=1$ for $|\xi|\geq 1$.
 Indeed, let  $\phi_1$ be an arbitrary   $ C^\infty$-function on $\mathbb{R}^n$, with $\phi_1(\xi)=0$ for $|\xi|\leq 1/2$ and $\phi_1(\xi)=1$ for $|\xi|\geq 1$. We write
$$m_\alpha(\xi)=|\xi|^{-\alpha}e^{i|\xi|}(\phi_1( \xi )-\phi( \xi))+|\xi|^{-\alpha}e^{i|\xi|}\phi(\xi)=m_\alpha^{(1)}+m_\alpha^{(2)}.$$
Since $m_\alpha^{(1)}$ is a $ C^\infty$-function with compact support, then the    corresponding   multiplier  is a bounded linear operator on $L^p(\mathbb{R}^n,w_k(x)dx)$, $1\leq p\leq \infty$.
  \par  Assume as  a first step  that   $\alpha\in]\gamma_k+(n-1)/2, \gamma_k+(n+1)/2[$. Here we shall prove that
  $T_{m_\alpha}$ is a  convolution operator with kernel $K_\alpha$ belongs to  $ L^1(\mathbb{R}^n,w_k(x)dx)$.
  \par We begin,  by writing via (\ref{fi})
 \begin{equation}\label{malpha}
    m_\alpha(\xi)=\sum_{\nu=0}^\infty m^\nu_\alpha(\xi),
    \end{equation}
  where  $m^\nu_\alpha(\xi)=\psi(2^{-\nu}|\xi|)|\xi|^{-\alpha}e^{\pm i|\xi|}$.
 Since the function $m^\nu_\alpha$ is radial, it follows that   $K^\nu_\alpha= \mathcal{F}_k^{-1}(m^\nu_\alpha) $ is radial and from (\ref{rad}),
  \begin{eqnarray*}
  K_\alpha^\nu(x) =d_k\int_{0}^\infty \psi(2^{-\nu}s)  s ^{-\alpha+2\gamma_k+n-1}e^{is} \mathcal{J}_{\gamma_k+n/2-1}(|x|s)ds.
\end{eqnarray*}
 In the next we  claim that the sum $\sum K_\alpha^\nu(x)$ is convergent for $|x|\neq 0,1$ and we have
\begin{eqnarray*}
 K_\alpha(x)= \sum _{\nu=0}^\infty  K_\alpha^\nu(x)= d_k\int_{0}^\infty \widetilde{ \phi}(s) s ^{-\alpha+2\gamma_k+n-1}e^{is} \mathcal{J}_{\gamma_k+n/2-1}(|x|s)ds
\end{eqnarray*}
 For this purpose we recall the appropriate asymptotic expansion of Bessel function (see \cite{Wat}),
 \begin{eqnarray}\label{bvB}\nonumber
 \mathcal{J}_{\gamma_k+n/2-1}(|x|s)&=&e^{i|x|s}\sum_{\ell=0}^{N-1}a_\ell (|x|s)^{-\gamma_k-(n-1)/2-\ell  }\\&&+e^{-i|x|s}\sum_{\ell=0}^{N-1}a'_\ell (|x|s)^{-\gamma_k-(n-1)/2-\ell  }+R_N(|x|s) ,
 \end{eqnarray}
where $a_\ell$ and $a'_\ell$ are constants and  the function $R_N$ satisfies the estimate
\begin{eqnarray}\label{RN}
 |R_N(t)|\leq c_N|t|^{-N-\gamma_k-(n-1)/2 }.
\end{eqnarray}
Thus for a large enough $N$, one  can write
\begin{eqnarray*}
K_\alpha^\nu(x)= d_k\sum_{\ell=0}^{N-1}
  |x|^{-\gamma_k-(n-1)/2-\ell  }\Big\{a_\ell\;  f_\alpha^{\nu,\,\ell} (1+|x|)+a_\ell' \;  f_\alpha^{\nu,\,\ell} (1-|x|)\Big\}+ d_k G_\alpha^{N,\,\nu}(x)
\end{eqnarray*}
where we put
\begin{eqnarray*}
f_\alpha^{\nu,\,\ell}(t)&=& \int_{0}^\infty \psi(2^{-\nu}s)  s^{-\alpha+ \gamma_k+ (n-1)/2-\ell }e^{its}ds, \\
G_\alpha^{N,\,\nu}(x)&= & \int_{0}^\infty \psi(2^{-\nu}s)  s ^{-\alpha+2\gamma_k+n-1}R_N(|x|s)\;e^{is} \;ds.
 \end{eqnarray*}
After possibly interchanging  sum and integral,  we have
\begin{eqnarray*}
\sum_{\nu=0}^\infty G_{N,\nu}(x) &= &   \int_{0}^\infty \widetilde{\phi}( s)  s ^{-\alpha+2\gamma_k+n-1}R_N(|x|s)\;e^{is} \;ds,\\
\sum_{\nu=0}^\infty f_{\nu,\ell}(t)&=& \int_{0}^\infty  \widetilde{\phi}(s) s^{-\alpha+ \gamma_k+ (n-1)/2-\ell }e^{its}ds; \quad \ell\geq 1.
 \end{eqnarray*}
When for  $\ell=0$ we shall need to integrate by parts,
\begin{eqnarray*}
&&\int_{0}^\infty \psi(2^{-\nu}s) s^{-\alpha+ \gamma_k+ (n-1)/2 }e^{its}ds=
 it^{-1}\int_{0}^\infty2^{-\nu} \psi'(2^{-\nu}s)  s^{-\alpha+ \gamma_k+ (n-1)/2 }e^{its}ds\\&&+i(-\alpha+ \gamma_k+ (n-1)/2)t^{-1}\int_{0}^\infty \psi(2^{-\nu}s)  s^{-\alpha+ \gamma_k+ (n-1)/2-1 }e^{its}ds.
 \end{eqnarray*}
The  interchange of summation being permissible because
 \begin{eqnarray*}
   \sum_{\nu=0}^\infty\int_{0 }^\infty|2^{-\nu} \psi'(2^{-\nu}s) s^{-\alpha+ \gamma_k+ (n-1)/2 }e^{its}| &=&
  \sum_{\nu=0}^\infty \int_{0}^\infty|2^{-\nu}s\; \psi'(2^{-\nu}s) |s^{-\alpha+ \gamma_k+ (n-1)/2 -1} ds\\
& \leq&  C  \sum_{\nu=0}^\infty\int_{2^{\nu-1}} ^{2^{\nu+1} }   s^{-\alpha+ \gamma_k+ (n-1)/2 -1} ds\\
&\leq & 2C\int_{1/2} ^{\infty }   s^{-\alpha+ \gamma_k+ (n-1)/2 -1} ds <\infty.
 \end{eqnarray*}
Thus,
 $$ \sum_{\nu=0}^\infty \int_{0}^\infty2^{-\nu} \psi'(2^{-\nu}s)  s^{-\alpha+ \gamma_k+ (n-1)/2 }e^{its}ds= \int_{0}^\infty \widetilde{\phi}'(s)  s^{-\alpha+ \gamma_k+ (n-1)/2 }e^{its}ds.$$
and
  \begin{eqnarray*}
\sum_{\nu=0}^\infty f_\alpha^{\nu,0}(t)&=& it^{-1}\int_{0}^\infty  \widetilde{\phi}'(s) s^{-\alpha+ \gamma_k+ (n-1)/2  }e^{its}ds
 \\&+& i(-\alpha+ \gamma_k+ (n-1)/2)t^{-1}
 \int_{0}^\infty \widetilde{ \phi}(s) s^{-\alpha+ \gamma_k+ (n-1)/2-1}e^{its}ds
   \\&=&   \int_{0}^\infty  \widetilde{\phi}(s) s^{-\alpha+ \gamma_k+ (n-1)/2  }e^{its}ds.
\end{eqnarray*}
We therefore conclude that
\begin{eqnarray}\label{DK}
\nonumber&&K_\alpha(x)=  d_k \sum_{\ell=0}^{N-1}
  |x|^{-\gamma_k-(n-1)/2-\ell  } \Big\{a_\ell \; \int_{0}^\infty  \widetilde{\phi}(s) s^{-\alpha+ \gamma_k+ (n-1)/2-\ell }e^{is}e^{i|x|s}ds\\&&+
  a_\ell' \; \int_{0}^\infty  \widetilde{\phi}(s) s^{-\alpha+ \gamma_k+ (n-1)/2-\ell }e^{is}e^{-i|x|s}ds.\Big\} +d_k
  \int_{0}^\infty  \widetilde{\phi}( s)  s ^{-\alpha+2\gamma_k+n-1}R_N(|x|s)\;e^{is} \;ds.\nonumber\\&&
\end{eqnarray}
and in view of (\ref{bvB})
\begin{eqnarray*}
 K_\alpha(x)= d_k\int_{0}^\infty \widetilde{ \phi}(s) s ^{-\alpha+2\gamma_k+n-1}e^{is} \mathcal{J}_{\gamma_k+n/2-1}(|x|s)ds.
 \end{eqnarray*}
  In order to study the behavior of $K_\alpha$ we will  need the following  elementary lemmas.
  \begin{lem} \label{l2} Let $\alpha>0$. The function given by
  $$h(t)=\int_0^\infty\phi(s)s^{-\alpha}e^{\pm its} ds, \qquad t\neq0$$
  satisfies  $h(t)=o(t^{-N})$ when  $t\rightarrow\infty$,  for any integer $N>0$.
  \end{lem}
 \begin{proof} By using integration by parts and Leibniz rule,
 \begin{eqnarray*}
    |t^Nh(t)|&=&\left|\int_0^\infty\left(\frac{d}{ds}\right)^N\Big(\phi(s)s^{-\alpha}\Big)e^{\pm its} ds\right|  \\
   \\&\leq & C\left\{\int _{1/2}^1\sum_{\ell=0}^N  |\;\phi^{(\ell)}(s)s^{-\alpha-N+\ell}| ds
    +  \int_1^\infty  s^{-\alpha-N }  ds \right\}
    \\ &\leq & C\;.
   \end{eqnarray*}
  \end{proof}
   \begin{lem}\label{maj}
    The following   are  continuous and  bounded functions on $(0,\infty)$
   $$G(t)=\int_0^t s^{-\alpha+\gamma_k+(n-1)/2}  e^{\pm is} \; ds,\;\;
   H(t)=t^{ -\alpha+\gamma_k+(n+1)/2}\int_1^2 \psi(s)s^{-\alpha+\gamma_k+(n-1)/2}   e^{\pm its} \; ds.$$
   \end{lem}
 \begin{proof}
 The continuous of $G$ and $H$ is  obvious. Since the integral defining $G$ is convergent then the  function G has a finite limit at $\infty $,
   from which  and    continuity   we obtain   the boundedness of $G$. Similarly for the boundedness of $H$, since by integration by parts  we have that
   $$\left|\int_1^2\psi(s) s^{-\alpha+\gamma_k+(n-1)/2} \; e^{\pm its} \; ds\right|\leq Ct^{-1},$$
  and then  $\lim_{t\rightarrow\infty} H(t)=0$.
  \end{proof}
 \par
  Let us now  consider the asymptotic behavior of   $K_\alpha$.
\begin{lem} The kernel $ K_\alpha$ has the following properties:
\begin{itemize}
  \item [(i)] $K_\alpha(x)=O(|x|^{-N}), \qquad as \; |x|\rightarrow\infty$, for big $N$,
  \item [(ii )]  $K_\alpha(x)=O\Big((1-|x|)^{\alpha- \gamma_k- (n-1)/2-1 }\Big)  , \qquad as  \; |x|\rightarrow 1$,
   \item [(iii )] $K_\alpha(x)$ is bounded near the origin,
  \item [(iv)]$K_\alpha$ is an integrable function.
\end{itemize}
\end{lem}
\begin{proof}
Clearly (iv) is a consequence of (i), (ii) and (iii). The proof of (i) follows from  (\ref{DK}), Lemma \ref{l2} and  (\ref{RN}).
  Consider (\ref{DK}) we  see   that  except  the term corresponding to $\ell=0$  all terms are bounded near $|x|= 1$.  On the other hand, writing

  \begin{eqnarray*}
  \int_{0}^\infty  \widetilde{\phi}(s) s^{-\alpha+ \gamma_k+ (n-1)/2}e^{i(1\pm|x|)s}ds&=&\int_{0}^\infty   s^{-\alpha+ \gamma_k+ (n-1)/2  }e^{i(1\pm|x|)s}ds\\&&+
\int_{0}^\infty  (\widetilde{\phi}(s)-1) s^{-\alpha+ \gamma_k+ (n-1)/2   }e^{i(1\pm|x|)s}ds
\\&=& (1\pm |x|)^{\alpha- \gamma_k- (n-1)/2-1}\int_{0}^\infty   s^{-\alpha+ \gamma_k+ (n-1)/2  }e^{i s}ds\\&&+
\int_{0}^\infty  (\widetilde{\phi}(s)-1) s^{-\alpha+ \gamma_k+ (n-1)/2   }e^{i(1\pm|x|)s}ds\\
 &=& C (1\pm |x|)^{\alpha- \gamma_k- (n-1)/2-1}+ H(x),
\end{eqnarray*}
where  $H$  is bounded. This gives the property (ii). To prove  (iii) we can use the well known integral representation  for Bessel function to write
\begin{eqnarray}\label{12}\nonumber
K_\alpha^\nu(x) &=& d_k\int_{0}^\infty \psi(2^{-\nu}s)  s ^{-\alpha+2\gamma_k+n-1}e^{is} \mathcal{J}_{\gamma_k+n/2-1}(|x|s)ds
  \\&=&b_k\int_{-1}^1(1-t^2)^{\gamma_k+(n-3)/2 }\left\{\int_{0}^\infty \psi(2^{-\nu}s)  s ^{-\alpha+2\gamma_k+n-1}
   e^{i(1+t|x|)s} \;ds \right\}\; dt,\nonumber \\&&
\end{eqnarray}
where $b_k=d_k \Gamma(\gamma_k+n/2)/\Gamma(\gamma_k+(n+1)/2-1)$. Let $|x|\leq 1/2 $ and  $N>-\alpha+2\gamma_k+n$, sufficiently large.
Using the fact that
$$   \left|\left(\frac{d}{ds}\right)^\ell (\psi (2^{-\nu}s) )\right|\leq C |s|^{-\ell},\quad \ell\in \mathds{N}$$
and applying   integration by parts $N$ times, to derive that
\begin{eqnarray*}
&& \int_{0}^\infty \psi(2^{-\nu}s)  s ^{-\alpha+2\gamma_k+n-1}e^{i(1+t|x|)s}\\&&= (i(1+t|x|))^{-N}
  \int_0^\infty\left(\frac{d}{ds}\right)^N\Big( \psi(2^{-\nu}s)  s ^{-\alpha+2\gamma_k+n-1}\Big)e^{i(1+t|x|)s} ds
     \\&&\leq  C2^{-N}\int _{2^{\nu-1}}^{2^{\nu+1}} s^{-\alpha+2\gamma_k+n-1-N }  ds.
   \end{eqnarray*}
From which and (\ref{12}) it follows that
  \begin{eqnarray}\label{bnu}
&&|K_\alpha^\nu(x)  |  \leq  C \int _{2^{\nu-1}}^{2^{\nu+1}} s^{-\alpha+2\gamma_k+n-1-N }  ds
   \end{eqnarray}
   and then  $|K_\alpha (x) |  \leq  C$.  This concludes   (iii).
   \end{proof}
\begin{lem}\label{rep}  Let $\alpha\in\;]\gamma_k+(n-1)/2, \gamma_k+(n+1)/2[$.
The operator $T_{m\alpha}$ can be represented  through the kernel $K_\alpha$ as the integral,
\begin{equation}\label{integrep}
   T_{m_\alpha}(f)(x)=\int_{\mathbb{R}^n} K_\alpha(y) \tau_x(f)(y)w_k(y)dy, \qquad f\in S(\mathbb{R}^n)
\end{equation}
and satisfies the inequality
\begin{equation}\label{L1}
  \|T_{m_\alpha}(f)\|_{1,k}\leq \| K_\alpha\|_{1,k}\|f\|_{1,k}.
\end{equation}
 \end{lem}
\begin{proof} Let $f\in S(\mathbb{R}^n)$. In view of (\ref{malpha}) we can
  write
$$T_{m_{\alpha}} =\sum_{\nu=0}^\infty  T_{m_\alpha^\nu}\:.$$
Each of     $T_{m_\alpha^\nu}$ will   be written in
its  integral form
\begin{eqnarray*}
T_{m_\alpha^\nu}(f)(x)= K_\alpha^\nu*_k f(x)=  \int_{\mathbb{R}^n}  K_\alpha^\nu(y) \tau_x(f)(-y)w_k(y)dy.
\end{eqnarray*}
 Hence to establish (\ref{integrep}) we only need to interchange the order of integration and summation.   For this purpose we split the integral
\begin{eqnarray*}
 &&\int_{\mathbb{R}^n}  K_\alpha^\nu(y) \tau_x(f)(y)w_k(y)dy =\int_{ |y|\leq 1/2}  K_\alpha^\nu(y) \tau_x(f)(y)w_k(y)dy\\&&\qquad\qquad+
 \int_{1/2\leq|y|\leq 2} K_\alpha^\nu(y) \tau_x(f)(y)w_k(y)dy  +\int_{|y|\geq 2}  K_\alpha^\nu(y) \tau_x(f)(y)w_k(y)dy.
\end{eqnarray*}
From the estimate (\ref{bnu}) we have
\begin{eqnarray*}
\sum_{\nu=0}^\infty \int_{ |y|\leq 1/2}  K_\alpha^\nu(y) \tau_x(f)(y)w_k(y)dy=\int_{ |y|\leq 1/2}  K_\alpha(y) \tau_x(f)(y)w_k(y)dy.
\end{eqnarray*}
Similarly for the  integral over $|y|\geq 2$, by writing
\begin{eqnarray*}
 &&K_\alpha^\nu(y) =\sum_{\ell=0}^{N-1}
  |y|^{-\gamma_k-(n-1)/2-\ell  }\Bigg\{c_\ell \; \int_{0}^\infty  \psi(2^{-\nu}s) s^{-\alpha+ \gamma_k+ (n-1)/2-\ell } e^{i(1+|y|)s}ds\\&&+
  c_\ell' \; \int_{0}^\infty \psi(2^{-\nu}s) s^{-\alpha+ \gamma_k+ (n-1)/2-\ell} e^{i(1-|y|)s}ds\Bigg\}+
  \int_{0}^\infty  \psi(2^{-\nu}s) s ^{-\alpha+2\gamma_k+n-1}R_N(|y|s)\;e^{is} \;ds,
  \end{eqnarray*}
for  $N$ sufficiently large, we get that
\begin{eqnarray*}
 &&|K_\alpha^\nu(y)|\leq C \int_{2^{\nu-1}}^{2\nu+1}  s ^{-\alpha+ \gamma_k+(n-1)/2-1}  \;ds.
  \end{eqnarray*}
Notice  that we have   integrate by parts   the integral of first  term ( for $\ell=0$ ). Thus,
\begin{eqnarray*}
\sum_{\nu=0}^\infty \int_{ |y|\geq2} K_\alpha^\nu (y) \tau_x(f)(y)w_k(y)dy =\int_{ |y|\geq 2}  K_\alpha(y) \tau_x(f)(y)w_k(y)dy.
\end{eqnarray*}
Now for the integral over  $1/2\leq|y|\leq 2$ we proceed by applying the dominated  convergence theorem. Put
  $$ \Psi_\nu(s)=\sum_{j=0}^{\nu}\psi(2^{-j}s), \qquad S_\alpha^\nu(y)=\sum_{j=0}^{\nu} K_\alpha^\nu(y).$$
First, observe that $\Psi_\nu(s)=1$ if $1\leq s\leq 2^\nu$ and  $\Psi_\nu(s)=0$ if $s\geq 2^{\nu+1}$ or $s\leq 1/2$.
Using this fact, we have
 \begin{eqnarray*}
 && \int_{0}^\infty  \Psi_\nu(s) s^{-\alpha+ \gamma_k+ (n-1)/2}e^{i(1\pm|y|)s}ds
 \\&& \qquad\qquad =  \int_{0}^1     \psi(s)s^{-\alpha+ \gamma_k+ (n-1)/2  }e^{i(1\pm|y|)s}ds +
\int_{1}^{2^\nu}  s^{-\alpha+ \gamma_k+ (n-1)/2   }e^{i(1\pm|y|)s}ds
\\&& \qquad\qquad  +\int_{2^\nu}^{2^{\nu+1}} \psi(2^{-\nu}s)s^{-\alpha+ \gamma_k+ (n-1)/2   }e^{i(1\pm|y|)s}ds.
\end{eqnarray*}
Clearly
  $$\left| \int_{0}^1   \psi(s)s^{-\alpha+ \gamma_k+ (n-1)/2  }e^{i(1\pm|y|)s}ds \right|\leq  \int_{0}^1     \psi(s)s^{-\alpha+ \gamma_k+ (n-1)/2  } ds.
   \leq C$$
  However,  Lemma \ref{maj} implies that
  \begin{eqnarray*}
   \left|\int_{1}^{2^\nu}  s^{-\alpha+ \gamma_k+ (n-1)/2   }e^{i(1\pm|y|)s}ds\right|&=&
   |1\pm|y||^{\alpha- \gamma_k -(n+1)/2}|\left|\int_{|1\pm|y||}^{2^\nu|1\pm|y||}  s^{-\alpha+ \gamma_k+ (n-1)/2   }e^{\pm is}ds\right|
   \\&\leq& C\; |1\pm|y||^{\alpha- \gamma_k -(n+1)/2}
   \end{eqnarray*}
and
 \begin{eqnarray*}
   \left|\int_{2^\nu}^{2^{\nu+1}} \psi(2^{-\nu}s)s^{-\alpha+ \gamma_k+ (n-1)/2   }e^{i(1\pm|y|)s}ds.\right|&=&
   |1\pm|y||^{\alpha- \gamma_k -(n+1)/2} H\Big((1\pm|y|)2^{\nu}\Big)
  \\& \leq &C\; |1\pm|y||^{\alpha- \gamma_k -(n+1)/2}.
   \end{eqnarray*}
  Now writing
\begin{eqnarray*}
&&S_\nu(y) =\sum_{p=0}^{N-1}
  |y|^{-\gamma_k-(n-1)/2-p  }\Bigg\{c_p \; \int_{0}^\infty \Psi_\nu( s) s^{-\alpha+ \gamma_k+ (n-1)/2-p } e^{i(1+|y|)s}ds\\&&+
  c_p' \; \int_{0}^\infty\Psi_\nu( s) s^{-\alpha+ \gamma_k+ (n-1)/2-p } e^{i(1-|y|)s}ds\Bigg\}+
  \int_{0}^\infty  \Psi_\nu( s) s ^{-\alpha+2\gamma_k+n-1}R_N(|y|s)\;e^{is} \;ds.
  \end{eqnarray*}
So, as  $1/2\leq|y|\leq 2$ we get that
 \begin{equation}\label{14}
    |S_\nu(y)|\leq C\; |1-|y||^{\alpha- \gamma_k -(n+1)/2}.
\end{equation}
Since the right hand side in (\ref{14}) is  an integrable function then by   the dominated convergence theorem,
 \begin{eqnarray*}
\sum_{\nu=0}^\infty \int_{1/2\leq |y|\leq 2} K_\alpha^\nu(y) \tau_x(f)(y)w_k(y)dy =\int_{ 1/2\leq |y|\leq 2}  K_\alpha(y) \tau_x(f)(y)w_k(y)dy.
\end{eqnarray*}
 This completes the proof of   (\ref{integrep}).
 \par To  prove  (\ref{L1}), we consider the truncated kernels and operators, for $0<\varepsilon<1$  $$K_{\alpha,\varepsilon}=K_\alpha \;\mathds{1} _{\{\varepsilon\leq ||y|-1|\leq \frac{1}{\varepsilon}\}} $$
 and for $f\in S(\mathbb{R}^n)$,
\begin{eqnarray*}
T_{m_\alpha,\varepsilon}f(x)=\int_{\mathbb{R}^n}K_{\alpha,\varepsilon}(y) \tau_x(f)(y)w_k(y)dy =\int_{\mathbb{R}^n} \tau_x(K_{\alpha,\varepsilon})(y) f(y)w_k(y)dy.
\end{eqnarray*}
Since the kernel $K_{\alpha }$ is in $L^1(\mathbb{R}^n,w_k(y)dy)$ and bounded a way from $|y|=1$ which imply that  $K_{\alpha,\varepsilon}$ is a bounded radial function in $L^1(\mathbb{R}^n,w_k(y)dy)$, we can then use (\ref{Y}) and  (\ref{trp}) to get the following
 $$ \|T_{m_\alpha,\varepsilon}(f)\|_{1,k}\leq \|\tau_x(K_{\alpha,\varepsilon})\|_{1,k}\|f\|_{1,k}\leq \|K_{\alpha,\varepsilon}\|_{1,k}\|f\|_{1,k}\leq
  \|K_{\alpha} \|_{1,k}\|f\|_{1,k}.$$
As $\lim_{\varepsilon\rightarrow \infty} T_{m_\alpha,\varepsilon}(f)(x)= T_{m_\alpha}(f)(x)$, then we can  apply Fatou's lemma to obtain (\ref{L1}).
\end{proof}
Now, we are in a position to prove Theorem \ref{th1}.
\begin{proof}[Proof of Theorem \ref{th1} ] We can assume that $1\leq p\leq 2$, the case $2\leq p\leq \infty$ follows by a well known duality argument.
 Let's choose an arbitrary $\alpha_0 \in\;]\gamma_k+(n-1)/2, \gamma_k+(n+1)/2[$. We first write
 \begin{eqnarray*}
  m_\alpha^\nu(\xi)&=& \Big\{\psi(2^{-\nu}|\xi|) |2^{-\nu}\xi|^{ \alpha_0-\alpha}\Big\} \Big\{ e^{i|\xi|} |\xi|^{-\alpha_0} \phi(|\xi| )\Big\}
  \\&=&2^{\nu(\alpha_0-\alpha)}h(2^{-\nu}\xi)  m_{\alpha_0}(\xi)
 \end{eqnarray*}
where $ h(\xi)= \psi( |\xi|) |\xi|^{\alpha_0-\alpha}$. The fact  that $h$ is a  radial $C^\infty$-function with compact support, the corresponding multiplier operator $T_h$    is  the convolution operator with the  radial Schwartz function $\mathcal{F}_k^{-1}(h)$ and   therefore
  bounded on   $L^r(\mathbb{R}^n, w_k(x)dx)$  for   all  $1\leq r\leq\infty$. This follows from (\ref{Y}). Also,
a simple   argument shows that   the multiplier    $T_{h_\nu}$ of symbol   $ h_{\nu}(\xi)=h( 2^{-\nu} \xi)$
is bounded on $L^r(\mathbb{R}^n, w_k(x)dx)$ with norm $\|T_{h_\nu}\|_r=\|T_{h }\|_r$.
Now, using Lemma \ref{L1} and that $T_{m_\alpha^\nu}$ is  a composition of two  multiplier operators on $L^1(\mathbb{R}^n, w_k(x)dx)$, it follows that  $T_{m_\alpha^\nu}$ is bounded on $L^1(\mathbb{R}^n, w_k(x)dx)$ with norm
\begin{eqnarray*}
 \|T_{m_\alpha^\nu}\|_1\leq C 2^{\nu(\alpha_0-\alpha)}.
\end{eqnarray*}
\par Similarly if we write
$ m_\alpha^\nu(\xi)= 2^{-\alpha\nu } \Big\{\psi(2^{-\nu}|\xi|) |2^{-\nu}\xi|^{- \alpha}\Big\} \Big\{ e^{i|\xi|}  \phi(|\xi| )\Big\}$ then we get that
\begin{eqnarray*}
 \|T_{m_\nu}\|_2\leq C 2^{-\alpha\nu }.
\end{eqnarray*}
 Therefore, by the Riesz-Thorin interpolation theorem, $T_{m_\alpha^\nu}$ is bounded  on   $L^p(\mathbb{R}^n, w_k(x)dx)$  for any $1\leq p\leq 2$,
$1/p=1-\theta/2$, $\theta\in [0,1]$, and we have that
\begin{eqnarray}\label{Tmp}
 \|T_{m_\alpha^\nu}\|_p\leq C 2^{\nu(\alpha_0-\alpha)(1-\theta)-\alpha\nu\theta} = C 2^{\nu[\alpha_0  (1-\theta)-  \alpha]}.
\end{eqnarray}
We can  then   sum and   obtain the boundedness of   $T_{m_\alpha}$   on   $L^p(\mathbb{R}^n, w_k(x)dx)$   if $\alpha>\alpha_0  (1-\theta)$.
Since,   $\alpha_0$ is arbitrary  in
$ ]\gamma_k+(n-1)/2,\; \gamma_k+(n+1)/2[$, it suffices to take
$$\alpha>  \gamma_k+(n-1)/2 (1-\theta)=(2\gamma_k+(n-1))\left(\frac{1}{p}-\frac{1}{2}\right) .$$
Our theorem is therefore proved.
\end{proof}
\subsection{Proof of Theorem \ref{th3}}
 We will use the following technical Lemma
 \begin{lem}\label{l5}
 Suppose that $F$  is \;$C^1(I)$, for  an interval $I$.
Then for each  $0<\lambda\leq |I|$  and $p,p' >1$ with  $1/p+1/p'=1$, we have
\begin{equation}\label{pp'}
  \sup_{I}|F(t)|\leq  \lambda^{-1/p }\left(\int_I|F(t)|^pdt\right)^{1/p}+\lambda^{1/p'}\left(\int_I|F'(t)|^pdt\right)^{1/p}
\end{equation}
 \end{lem}
\begin{proof}
Let $I_0\subset I$ be an interval of length $\lambda$. For $t,s\in I_0$ we have
$$F(t)=F(s)+\int_{s}^{t}F'(u)du.$$
So by H\"{o}lder's inequality
$$|F(t)|\leq |F(s)|+\lambda^{1/p'}\left(\int_I|F'(t)|^pdt\right)^{1/p}.$$
Integrating both sides with respect to $s$, we get
$$|F(t)|\leq\lambda^{-1}\int_I|F(s)|ds+\lambda^{1/p'}\left(\int_I|F'(t)|^pdt\right)^{1/p}$$
and  using again  Holder's inequality to obtain
$$|F(t)|\leq \lambda^{-1/p}\left(\int_I|F(t)|^pdt\right)^{1/p}+\lambda^{1/p'}\left(\int_I|F'(t)|^pdt\right)^{1/p}.$$
Then (\ref{pp'}) is established.
\end{proof}
 The next lemma is essentially contained in the  theorems 6.2 and 6.1 of  \cite{Xu}.
 \begin{lem}\label{l8}
    Let $ \Phi \in L^1(\mathbb{R}^n w_k(x)dx)$ be a real valued radial function with $ \mathcal{F}_k(\Phi) \in L^1(\mathbb{R}^n, w_k(x)dx)$  and which satisfies
$|\Phi(x)|  \leq C(1 + |x|)^{-2\gamma_k-n-1}$. Then we have
  $$\|\sup_{t>0}\Phi_t*_kf(x)\|_{p,k}\leq C \|f\|_{p,k};\quad  1< p\leq \infty,$$
  where $\Phi_t(x)=t^{-2\gamma_k-n}\Phi(t^{-1}x)$.
 \end{lem}
Now  we consider our  maximal  operator
 $$A_\alpha (f)(x)=\sup_{ t>0}|A_\alpha(f)(x,t)|$$
 where
 $$A_\alpha (f)(x,t)=\int_{ \mathbb{R}^n} \frac{e^{\pm it|\xi|}}{|t\xi|^\alpha}\phi(t\xi)\mathcal{F}(f) (\xi)E_k(ix,\xi)w_k(\xi)d\xi.$$
We  write
$$A_\alpha (f)(x,t) = \sum_{\nu=0}^\infty A_\alpha ^\nu(f)(x,t)$$
where
$$A_\alpha^\nu (f)(x,t)=\int_{ \mathbb{R}^n} \frac{e^{\pm it|\xi|}}{|t\xi|^\alpha}\phi(t\xi)\psi(2^{-\nu}t|\xi|)\mathcal{F}(f) (\xi)E_k(ix,\xi)w_k(\xi)d\xi.$$
Let $t\in[1,2]$. We can use  the  same argument that  provided   (\ref{Tmp})  and duality  to get the following estimates
$$\|A_\alpha^\nu (f)(.,t))\|_{p,k}\leq C 2^{\nu[\alpha_0  |2/p-1|-  \alpha]}\|f\|_{p,k}$$
$$\|\frac{\partial}{\partial t}A_\alpha^\nu (f)(.,t)\|_{p,k}\leq C 2^{\nu[\alpha_0 |2/p-1|- ( \alpha-1)]}\|f\|_{p,k}$$
 for an arbitrary $\alpha_0\in ]\gamma_k+(n-1)/2,\;\gamma_k+(n+1)/2$ and $  p\geq1$. Now by  applying the lemma \ref{l5} with $\ell=2^{-\nu}$ it follows that
 \begin{equation}\label{122}
    \|\sup_{1\leq t\leq2}A_\alpha^\nu (f)(.,t)\|_{p,k} \leq C 2^{\nu[\alpha_0 |2/p-1|-  \alpha+1/p]}\|f\|_{p,k}\;.
 \end{equation}
In addition, (\ref{122}) implies that  for $j\in \mathbb{Z}$,
    $$\|\sup_{2^j\leq t\leq2^{j+1}}A_\alpha^\nu (f)(.,t)\|_{p,k} \leq C 2^{\nu[\alpha_0 |2/p-1|-  \alpha+1/p]}\|f\|_{p,k}$$
which  can be seen by  writing
\begin{equation}\label{122'}
\sup_{2^{j}\leq t\leq2^{j+1}}|A_\alpha^\nu (f)(x,t)|=\sup_{1\leq t\leq2}|A_\alpha^\nu (f(2^{j}.))(2^{-j}x,t)|.
\end{equation}
In the next we claim that for $p\geq 2$,
\begin{equation}\label{123}
\|\sup_{t>0}|A_\alpha^\nu (f)(.,t)\|_{p,k}\leq C \;2^{\nu [\alpha_0 |2/p-1|-  \alpha+1/p]}  \|f\|_{p,k}
\end{equation}
 which asserts that
 \begin{equation}\label{123'}
\|\sup_{t>0}|A_\alpha  (f)(.,t)\|_{p,k}\leq C \;  \|f\|_{p,k},
\end{equation}
for  $p\geq2$ and $\alpha>  (2\gamma_k+(n-1))\left(1/2-1/p\right)+1/p$.
 For the proof we invoke a variant of the Littlewood-Paley  operators associated to  Dunkl transform ,  see Corollary 4.2 of  \cite{Dai}. Let  $t\in[2^j,2^{j+1}]$, $j\in \mathbb{Z}$. Define the function $f_\ell$ by $ \mathcal{F}_k(f_\ell)(\xi)=\mathcal{F}_k(f)(\xi)\psi(2^{-\ell}|\xi|)$,
 $\ell\in \mathbb{Z}$. It follows  that
$$A_\alpha ^\nu(f )(x,t) =A_\alpha ^\nu\left(\sum_{  |\ell+j-\nu| \leq 2}f_\ell\right)(x,t). $$
 From this and   (\ref{122'}) we have  for all $p\geq2$
 \begin{eqnarray*}
 \int_{\mathbb{R}^n}\sup_{t>0}|A_\alpha^\nu (f)(x,t)| ^p w_k(x)dx&\leq&
 \sum_{j\in\mathbb{Z}} \int_{\mathbb{R}^n}\sup_{2^j\leq t\leq2^{j+1}}|A_\alpha^\nu (f)(x,t)|^pw_k(x)dx
     \\ &\leq& C\;2^{\nu p\;[\alpha_0 |2/p-1|-  \alpha+1/p]} \sum_{j\in\mathbb{Z}} \int_{\mathbb{R}^n}\sum_{  |\ell+j-\nu| \leq 2}|f_\ell(x)|^pw_k(x)dx
     \\ &\leq& C\;2^{\nu p\;[\alpha_0 |2/p-1|-  \alpha+1/p]}  \int_{\mathbb{R}^n}\sum_{ \ell\in\mathbb{Z}}|f_\ell(x)|^pw_k(x)dx
                    \\ &\leq &C \;2^{\nu p[\alpha_0 |2/p-1|-  \alpha+1/p]} \int_{\mathbb{R}^n}
                     \left(\sum_{ \ell\in\mathbb{Z}}|f_\ell(x)|^2\right)^{p/2}w(x)dx\\
                     &\leq& C \;2^{\nu p[\alpha_0 |2/p-1|-  \alpha+1/p]}  \|f\|_{p,k}^p.
\end{eqnarray*}
This yields the claim (\ref{123}).
\par The  range $p\leq  2$ follows by interpolation argument, we proceed as in the proof of the theorem \ref{th2}. Let us first
 observe that when $\alpha>\gamma_k+(n+1)/2$ the kernel $K_\alpha$ of the operator $T_{m_\alpha}$ satisfies the following decay estimates
 $$|K_\alpha(x)|\leq \frac{C_N}{(1+|x|)^N}; \quad n\in \mathbb{N},\qquad N\in \mathbb{N}$$
 which is  immediately  by making  use of (\ref{rad}), (\ref{bvB}) and integration by parts.
  Hence using Lemma \ref{l8}  we get
 the boundedness of $A_\alpha$ on $L^p(\mathbb{R}^n,w(x)dx)$  for $1 < p \leq \infty$.
 The key tool is the following.  For  an arbitrary $\alpha_0>\gamma_k+(n+1)/2$, one  can write
  \begin{eqnarray*}
  m_\alpha^\nu(t\xi)&=& \Big\{\psi(2^{-\nu}t|\xi|) |2^{-\nu}t\xi|^{ \alpha_0-\alpha}\Big\} \Big\{ e^{it|\xi|} |t\xi|^{-\alpha_0} \phi(t\xi)\Big\}
  \\&=&2^{\nu(\alpha_0-\alpha)}h(2^{-\nu}t\xi)  m_{\alpha_0}(t\xi)
 \end{eqnarray*}
where $ h(\xi)= \psi( |\xi|) |\xi|^{\alpha_0-\alpha}$. Put $H=\mathcal{F}_k^{-1}(h)$, which is a radial Schwartz function. Thus we can write
$$A_\alpha ^\nu (f)(x,t)=2^{\nu(\alpha_0-\alpha)}H_{2^{\nu}t^{-1}}*_k A_{\alpha_0}  (f)(.,t)(x)$$
and we have that
$$|A_\alpha ^\nu (f)(x,t)|\leq 2^{\nu(\alpha_0-\alpha)}|H_{2^{\nu}t^{-1}}|*_k |A_{\alpha_0}  (f(.,t)|(x)
 \leq 2^{\nu(\alpha_0-\alpha)}|H|_{2^{\nu}t^{-1}}*_k |A_{\alpha_0}  (f)|(x).$$
Then  from Lemma  \ref{l8} and the boundedness of $A_{\alpha_0}$ we get
 $$\|\sup_{t>0}|A_\alpha^\nu  (f)(.,t)\|_{q,k}\leq  C 2^{\nu(\alpha_0-\alpha)}\|f\|_{q,k}; \quad 1<q\leq \infty.$$
On the other hand, the boundedness on $L^2(\mathbb{R},w_k(x)dx)$ in (\ref{123}) gives
  $$\|\sup_{t>0}|A_\alpha^\nu  (f)(.,t)\|_{2,k}\leq  C 2^{\nu(1/2- \alpha)}\|f\|_{2,k}.$$
So, by using the Riesz-Thorin interpolation Theorem,
\begin{eqnarray}\label{Tp}
 \|\sup_{t>0}|A_\alpha^\nu  (f)(.,t)\|_p\leq C \; 2^{\nu (\theta(\alpha_0-\alpha)+(1-\theta)(1/2- \alpha))}=C\;2^{\nu (\theta(\alpha_0-1/2)+ 1/2- \alpha)}
\end{eqnarray}
for all   $1<q\leq p\leq 2$, where $1/p=1/2+\theta(1/q-1/2)$   . It is now easy to check from this that the sum $\sum_\nu\|\sup_{t>0}|A_\alpha^\nu  (f)(.,t)\|_p$ is finite
 when $$\alpha>(\frac{1}{p}-\frac{1}{2})( 2\gamma_k+n-1 )+ \frac{1}{p}.$$
In fact, as $\theta=(1/p-1/2)(1/q-1/2)^{-1}$ we see that
$\theta(\alpha_0-1/2)+ 1/2$ tends to $(1/p-1/2)( 2\gamma_k+n-1 )+ 1/p$ when letting $\alpha_0$ go to $\gamma_k+(n+1)2$ and $q$ go to $1$
 which guarantees the choice of   $\alpha_0$   and $q$
such that $\alpha>\theta(\alpha_0-1/2)+ 1/2$.  This conclude the case $p\leq 2$ and
thus we have completely proved our theorem.
  \subsection{Proof of Theorem \ref{th4}}
 Let $\alpha>0$.  As a first step we write
 \begin{eqnarray*}
    \mathcal{J}_{\alpha+\gamma_k+n/2-1}(t|\xi|)&= &\mathcal{J}_{\alpha+\gamma_k+n/2-1}(t|\xi|)\phi(t \xi )+ \mathcal{J}_{\alpha+\gamma_k+n/2-1}(t|\xi|)(1-\phi(t\xi))\\
    &=& a_\alpha^{(1)}(t\xi)+ a_\alpha^{(2)}(t\xi)
\end{eqnarray*}
and so, $\mathcal{M}_\alpha(f) =\mathcal{M}_\alpha^{(1)} +\mathcal{M}_\alpha^{(1)} $  with
$$\mathcal{M}_\alpha^{(i)}(f) = \sup_{t>0}\left|\int_{\mathbb{R}^n} a_\alpha^{(i)}(t|\xi|) \mathcal{F}_k(f)(\xi)w_k(\xi)d\xi\right|=
\sup_{t>0}|\mathcal{M}_\alpha^{(i)}(f)(x,t)|, \quad i=1,2.$$
 Since the function  $a_\alpha^{(2)}$  is a $C^\infty$ with compact support, in view of Lemma \ref{l8}  we get the $L^p$-boundedness of
 $\mathcal{M}_\alpha^{(2)} $ for all $1<p\leq \infty$. To treat the boundedness of $\mathcal{M}_\alpha^{(1)}$ we invoke the asymptotic form (\ref{bvB}) for the Bessel function $\mathcal{J}_{\alpha+\gamma_k+n/2-1}$. For $N$ big enough,
\begin{eqnarray*}
 \mathcal{J}_{\alpha+\gamma_k+n/2-1}(t|\xi| )&=&e^{it|\xi|}\sum_{\ell=0}^{N-1}c_\ell (t|\xi|)^{-\alpha-\gamma_k-(n-1)/2-\ell  }\\&&+e^{-it|\xi|}\sum_{\ell=0}^{N-1}c'_\ell (t|\xi|)^{-\alpha-\gamma_k-(n-1)/2-\ell  }+R_N(t|\xi|),
 \end{eqnarray*}
where $c_\ell$ are constants. The function $R_N$ satisfies for all $j \in \mathbb{N}$
\begin{eqnarray}\label{R}
 |\partial_s^jR_N(s)|\leq C|s|^{-\alpha-N-\gamma_k-(n-1)/2 }.
\end{eqnarray}
It is not hard to show that this requirement is satisfied.
Set $V =\mathcal{F}_k^{-1}(R_N(|.|)\phi(.))$, which  can be represented  by
\begin{eqnarray*}
 V(x) = \frac{2^{-\gamma_k-n/2+1}}{\Gamma(\gamma_k+n/2)}\int_{0}^\infty R_N(s)\widetilde{\phi}(s)  s ^{ 2\gamma_k+n-1}  \mathcal{J}_{\gamma_k+n/2-1}(|x|s)ds.
\end{eqnarray*}
So, with the use  of asymptotic form (\ref{bvB}), integration by part and (\ref{R}) we get that for some larger $N$,
$$|V(x)|\leq C (1+|x|)^{-N}$$
and because of  Lemma \ref{l8}
 $$\|\sup_{t>0}V_t*_kf \|_{p,k} \leq C_p \|f\|_{p,k}; \quad 1< p\leq \infty.$$
It therefore  only remains to check  the needed conditions for the boundedness of   the multipliers $A_{ \alpha+\gamma_k+(n-1)/2+\ell}$, for  $\ell=0,1...N-1$.
  But this follow from Theorem \ref{th3} and can be reduced to  the $L^p$- boundedness  of   $A_{ \alpha+\gamma_k+(n-1)/2 }$  that is
 $$\alpha+\gamma_k+(n-1)/2>( 2\gamma_k+n-1 )\left|\frac{1}{p}-\frac{1}{2}\right|+ \frac{1}{p}$$
or equivalently
 \begin{itemize}
   \item[(a)] $\alpha  > 1-2\gamma_k-n+(2\gamma_k+n)/p$, $\qquad$if  $ \qquad 1< p \leq 2$
   \item[(b)] $\alpha> (2-2\gamma_k-n)/p$, $\qquad$ if $\qquad 2\leq p <\infty,$
 \end{itemize}
which is the desired statement   of Theorem \ref{th4}.

 \end{document}